\documentclass[11pt,english]{article}
\usepackage{lmodern}

\usepackage[T1]{fontenc}
\usepackage[a4paper]{geometry}
\usepackage{mathrsfs}
\usepackage{amsmath}
\usepackage{amsthm}
\usepackage{amssymb}
\usepackage{stackrel}
\usepackage{setspace}
\usepackage{esint}
\PassOptionsToPackage{normalem}{ulem}
\usepackage{ulem}
\makeatletter
\theoremstyle{plain}
\newtheorem{thm}{\protect\theoremname}
\theoremstyle{definition}
\newtheorem{problem}[thm]{\protect\problemname}
\theoremstyle{remark}
\newtheorem*{rem*}{\protect\remarkname}
\theoremstyle{plain}
\newtheorem{prop}[thm]{\protect\propositionname}
\theoremstyle{plain}
\newtheorem{lem}[thm]{\protect\lemmaname}

\usepackage{babel}

\makeatother

\usepackage{babel}
\providecommand{\lemmaname}{Lemma}
\providecommand{\problemname}{Problem}
\providecommand{\propositionname}{Proposition}
\providecommand{\remarkname}{Remark}
\providecommand{\theoremname}{Theorem}

\begin{document}
\title{{\Large{}Restriction of the Hamilton-Jacobi equation to a submanifold
$M$ of $\mathbb{R}^{d}$: }
{\Large{}case of $M\times\mathbb{R}^{d}$ invariant by the Hamiltonian}}
\author{{\Large{}Author: Othmane ISLAH}}
\maketitle

\section{Introduction}

\subsection{Set-up of the problem and preliminaries}

We are interested in this article in the question of restricting solutions
of the Hamilton-Jacobi equation to a submanifold $M$ of dimension
$m$ of $\mathbb{R}^{d}$, assumed to be a closed subset of $\mathbb{R}^{d}$
and without boundary but not necessarily compact and connected ($m,d\in\mathbb{\mathbb{N}}$).
We recall first the general set-up of the problem which is similar to \cite{BER} and \cite{ROO}.  Then we will
formulate more precisely the question that needs to be tackled.

Let $H:\,T\mathbb{R}^{d}\longrightarrow\mathbb{R}$ be a Hamiltonian
function of class $C^{2}$ and $u_{0}:\mathbb{R}^{d}\longrightarrow\mathbb{R}$
a Lipschitz function. If $T\geq0$ , let $I=[0,T]$ or $\mathbb{R}_{+}$,
and $\Omega=I\mathbb{\times R}^{d}$. We denote by $\mathscr{(CP)}$,
the following Cauchy problem i.e the evolutionary Hamilton-Jacobi
equation: 
\[
\begin{cases}
\forall\left(t,q\right)\in\Omega & \partial_{t}u(t,q)+H\left(t,\nabla_{q}u(t,q)\right)=0\\
\forall q\in\Omega & u(0,q)=u_{0}(q)
\end{cases}
\]

We will consider for $\mathscr{(CP)}$, both strong solutions (differentiable
functions that are solutions in the classical sense) and locally Lipschitz
solutions in the viscosity sense. If $u:\Omega\longrightarrow\mathbb{R}$
is a solution to $\mathscr{(CP)}$, let $\bar{u}=u_{\mid I\times M}$
be the restriction of $u$ to $I\mathbb{\times}M$ and similarly let
$\bar{u}_{0}=u_{0\mid M}$ . Since $TM$ the tangent bundle of $M$,
is an embedded submanifold of $T\mathbb{R}^{d}$, we can define the
restriction $\overline{H}=H{}_{\mid TM}$. 
\begin{problem}
Under what conditions, can we say that $\mathrm{\overline{u}}$ is
a solution of the following Cauchy problem $\mathscr{(\overline{CP})}$
, i.e the restricted evolutionary Hamilton-Jacobi equation, 
\[
\begin{cases}
\forall\left(t,q\right)\in I\mathbb{\times}M & \partial_{t}\bar{u}(t,q)+\overline{H}\left(t,\nabla_{q}\bar{u}(t,q)\right)=0\\
\forall q\in M & \bar{u}(0,q)=\bar{u}_{0}(q)
\end{cases}
\]
? 
\end{problem}

Let $\mathrm{\varphi_{H}}$ be the Hamiltonian flow of the Hamiltonian
vector field $X_{H}$ given for $(q,p)\in\mathbb{R}^{d}$ 
\[
\mathrm{X}_{\mathrm{H}}(q,p)=(\nabla_{p}H(q,p),-\nabla_{q}H(q,p))
\]
We also fix as a norm, the canonical euclidean norm on $T\mathbb{R}^{d}$.We
make sometime in this article a set of assumptions on $H$ that there
exists $C>0$ so that for $(q,p)\in T\mathbb{R}^{d}$ 
\begin{align}
\left\Vert \mathrm{d^{2}}H(q,p)\right\Vert \leq C,\,\,\left\Vert \nabla_{p}H(q,p)\right\Vert \leq C\left(1+\left\Vert p\right\Vert \right),\,\, & \left|H(q,p)\right|\leq C\left(1+\left\Vert p\right\Vert ^{2}\right)\label{eq:Hamiltonian assumptions}
\end{align}

We will show in this article that the answer to this question is affirmative
under following assumptions on $H$, $M$ and $u_{0}$

if $M\times\mathbb{R}^{d}$ is invariant by $\mathrm{\varphi}_{H}$
then $\bar{u}$ is a classical solution (respectively, a viscosity
solution) of $\mathscr{(\overline{CP})}$ , if $u$ is is a classical
solution (respectively, a viscosity solution) of $\mathscr{(CP)}$

\subsection{Main result}

We need to recall some basic geometric facts. $M$ being a submanifold
of $\mathbb{R}^{d}$, we consider the normal bundle $NM$ of $M$,
and we have by definition: 
\[
NM=\underset{q\in M}{\bigcup}\{q\}\times(T_{q}M)^{\bot}
\]
The Tubular Neighbourhood Theorem ensures there exists a continuous
positive function on $M$, $\theta:\,q\rightarrow\theta(q)$ such
that if $N_{\theta}=\{(q,n)\in NM/\left\Vert n\right\Vert <\theta(q)\}$
then : 
\begin{eqnarray*}
\psi: & N_{\theta}\longrightarrow\mathbb{R}^{d}\\
 & (q,n)\longrightarrow q+n
\end{eqnarray*}

is a smooth diffeomorphism onto onto its open image $U_{\theta}=\psi(N_{\theta})$

Let $\pi_{1}:\,NM\longrightarrow M$ ~be the canonical projection.
For $q\in U_{\theta}$, let $\tilde{q}=\pi_{1}\circ\psi^{-1}(q)$
and $v(q)=d(\pi_{1}\circ\psi^{-1})(q)|_{T_{\tilde{q}}M}$. We prove
first that $v(q)$ is an automorphism of $T_{\tilde{q}}M$. It is
enough to prove that $v(q)$ is injective. First, we have: 
\[
d(\pi_{1}\circ\psi^{-1})(q)=d\pi_{1}(\tilde{q})\circ d\psi^{-1}(q)\,\tag{chain rule}
\]
This linear map is surjective so by the rank-nullity theorem: 
\[
\dim(\ker(d(\pi_{1}\circ\psi^{-1})(q)))=d-m
\]
Since $(T_{\tilde{q}}M)^{\perp}\subset\ker(d(\pi_{1}\circ\psi^{-1})(q))$
and since these two vector spaces have the same dimension, we deduce
that : 
\[
(T_{\tilde{q}}M)^{\perp}=\ker(d(\pi_{1}\circ\psi^{-1})(q))
\]
So, $v(q)$ is injective and therefore an automorphism since it is
an endomorphism of $T_{\tilde{q}}M$. We will see in the next section
further properties of the automorphism $v(q)$.

We will also denote for $\,q\in M$, $\varPi_{q}$ the symmetric matrix
of $\mathcal{M}_{d}(\mathbb{R})$ which is the orthogonal projector
on the subvector space $T_{q}M\subset\mathbb{R}^{d}$. 

We come now to the precise statement of the first main result:
\begin{thm}
\label{thm:Restrictions-and-extensions}Restrictions and extensions
under invariance of $M\times\mathbb{R}^{d}$

a) If $M\times\mathbb{R}^{d}$ is invariant by $\mathrm{\varphi_{H}}$
and if $u$ is a classical solution (respectively, a locally Lipschitz
viscosity solution) of $\mathscr{(CP)}$ then $\bar{u}$ is a classical
solution (respectively a viscosity solution) of $\mathscr{(\overline{CP})}$
, 

b) Let us assume we are given $C^{1}$ functions $\bar{H}:TM\longrightarrow\mathbb{R}$
and $\bar{u}_{0}:M\longrightarrow\mathbb{R}$ such that there exits
$\overline{u}:I\times M\longrightarrow\mathbb{R}$ which is a differentiable
solution (respectively, a locally Lipschitz viscosity solution) of
$\mathscr{(\overline{CP})}$. We define : 
\begin{align*}
H:TU_{\theta} & =U_{\theta}\times\mathbb{R}^{d}\longrightarrow\mathbb{R}\\
 & (q,p)\longmapsto\bar{H}\left(\tilde{q},v^{-1}(q)(\varPi_{q}p)\right)
\end{align*}

(where $\tilde{q}=\pi_{1}\circ\psi^{-1}(q)$ and $v(q)$ as above).
Also, for $a\in\mathbb{R}$, we define for $(t,q)\in I\times U_{\theta}$
\begin{itemize}
\item $u_{0}(q)=\bar{u}_{0}(\tilde{q})+a\left\Vert q-\tilde{q}\right\Vert ^{2}$ 
\item $u(t,q)=\bar{u}(t,\tilde{q})+a\left\Vert q-\tilde{q}\right\Vert ^{2}$
\end{itemize}
Then $u$ is is a differentiable solution (respectively, a locally
Lipschitz viscosity solution) of $\mathscr{(CP)}$) on $I\times U_{\theta}$
with Hamiltonian $H$ and initial condition $u_{0}$. 
\end{thm}

We will verify in the following section that the invariance of $M\times\mathbb{R}^{d}$
by $\varphi_{H}$ is equivalent to the following property of independence
of the Hamiltonian function from normal moments on $M$: 
\[
\forall(q,p)\in M\times\mathbb{R}^{d}\,\,H(q,p)=H(q,\varPi_{q}p)
\]

\begin{rem*}
In other words, the second part of this theorem asserts that starting
from a solution $\bar{u}$ of the Hamilton-Jacobi equation defined
on $\mathbb{R}\times M$ with Hamiltonian function $\overline{H}:\,TM\rightarrow\mathbb{R}$
, we can extend it to a function $u$ defined in a neighbourhood $U$
of $M$ such that it is a solution of the same type of the Hamilton-Jacobi
equation on $\mathbb{R}\times U$ for a certain Hamiltonian function
$H:\,TU\rightarrow\mathbb{R}$. Moreover, $M\times\mathbb{R}^{d}$
is invariant by the Hamiltonian flow $\varphi_{H}$. 
\end{rem*}
We notice that the previous theorem cannot be applied if $H$ is fiberwise
strictly convex in $p$, simply because in this case $M\times\mathbb{R}^{d}$
cannot be invariant by the Hamiltonian flow since the invariance requires
an independence from normal moments.

We will first interpret in section \eqref{sec:Geometry-of-the} geometrically
and with respect to the Hamiltonian vector field the assumption of
invariance of $M\times\mathbb{R}^{d}$as well as $TM$ . We will also
work out some other geometric properties of the problem at hand. The
section \eqref{sec:Restriction-of-solutions} will be devoted to the
proof of Theorem \eqref{thm:Restrictions-and-extensions} and the
methods used are essentially classical. 

\section{\label{sec:Geometry-of-the}Geometry of the problem and the dual
point of view }

Throughout this article, the Euclidean scalar product on $\mathbb{R}^{d}\times\mathbb{R}^{d}$
will be denoted by the brackets$<,>$. Also, for $q\in M$, the orthogonal
projector on $T_{q}M$ (respectively on $\left(T_{q}M\right)^{\perp}$)
will be denoted by the symmetric metric $\varPi_{q}$ (respectively
$\varPi_{q}^{\perp}$).

\subsection{Invariance of $M\times\mathbb{R}^{d}$}

As mentioned above, we show explicitly here the meaning of the invariance
of $M\times\mathbb{R}^{d}$ in terms of the behaviour of the Hamiltonian
on $M\times\mathbb{R}^{d}$ 
\begin{prop}
\label{prop:invariance M cross Rn }$M\times\mathbb{R}^{d}$ is invariant
by the Hamiltonian flow if and only if : 
\[
\forall q\in M,\,\forall p\in\mathbb{R}^{d}\,H(q,p)=H(q,\Pi_{q}(p))
\]

where $\Pi_{q}$ is the orthogonal projection on $T_{q}M$. 
\end{prop}

\begin{proof}
First we notice that since $M\times\mathbb{R}^{d}$ is a closed subset
of $T\mathbb{R}^{d}$, the invariance by the Hamiltonian flow of is
equivalent to : 
\[
\forall(q,p)\in M\times\mathbb{R}^{d},\,(\nabla_{p}H(q,p),-\nabla_{q}H(q,p))\in T_{(q,p)}M\times\mathbb{R}^{d}=T_{q}M\times\mathbb{R}^{d}
\]

which is the same as:: 
\[
\forall(q,p)\in M\times\mathbb{R}^{d},\,\nabla_{p}H(q,p)\in T_{q}M
\]

Let us prove the only if part first. We have for $(q,p)\in M\times\mathbb{R}^{d}$
: 
\begin{align*}
H(q,p)= & H(q,\Pi_{q}(p))+\int_{0}^{1}\frac{d}{ds}H(q,\Pi_{q}(p)+s(p-\Pi_{q}(p))).ds\\
= & H(q,\Pi_{q}(p))+\int_{0}^{1}<\nabla_{p}H(q,\Pi_{q}(p)+s(p-\Pi_{q}(p))),\,p-\Pi_{q}(p)>ds
\end{align*}

But : 
\[
<\nabla_{p}H(q,\Pi_{q}(p)+s(p-\Pi_{q}(p))),\,p-\Pi_{q}(p)>=0
\]
because by assumption on $H$, 
\[
\nabla_{p}H(q,\Pi_{q}(p)+s(p-\Pi_{q}(p)))\in T_{q}M
\]
and 
\[
p-\Pi_{q}(p)\in(T_{q}M)^{\bot}
\]

So for $(q,p)\in M\times\mathbb{R}^{d}$ 
\[
H(q,p)=H(q,\Pi_{q}(p))
\]

For the if part, it is enough to notice that for $(q,p)\in M\times\mathbb{R}^{d}$,
and $(t,n)\in\mathbb{R}\times(T_{q}M)^{\bot}$, 
\[
H(q,p+tn)=H(q,\Pi_{q}(p))
\]
So for any $n\in(T_{q}M)^{\bot}$: 
\[
<\nabla_{p}H(q,p),n>=0
\]
which implies that $\nabla_{p}H(q,p)\in T_{q}M$. 
\end{proof}

\subsection{Second fundamental form of $M$ and invariance of $TM$}

We are going to recall the characterization of the double tangent
space $TTM$ in terms of the second fundamental form of the submanifold
$M$ with respect to the Euclidean metric.This will allow us to obtain
explicit formula for the invariance of $TM$ by the Hamiltonian flow.
The results here are borrowed from \cite{RS}.First, let us recall some definitions
and classical properties. For each $q\in M$, the orthogonal projection
on $T_{q}M$, is given by a symmetric matrix $\varPi_{q}\in\mathcal{M}_{n}\left(\mathbb{R}\right)=\mathbb{R}^{n\times n}$
and the following map is smooth: 
\begin{align*}
\varPi:M & \rightarrow\mathbb{R}^{n\times n}\\
q & \rightarrow\varPi_{q}
\end{align*}

Then for $v\in T_{q}M$, $d\Pi_{q}(v)\in\mathbb{R}^{n\times n}$ so
as a matrix it can be multiplied by a vector. Moreover we have, $\forall q\in M,\forall v,w\in T_{q}M$:
\[
\,d\varPi_{q}(v)w=d\varPi_{q}(w)v\in(T_{q}M)^{\bot}
\]

The collection of symmetric bilinear maps defined for $q\in M$ by
: 
\begin{align*}
h_{q}:T_{q}M\times T_{q}M & \rightarrow(T_{q}M)^{\bot}\\
(v,w) & \mapsto d\varPi_{q}(v)w
\end{align*}

is called the second fundamental form on $M$.

Finally the tangent space at a point $(q,p)\in TM$ can be expressed
in terms of the second fundamental form as : 
\[
T_{(q,p)}TM=\left\{ (\hat{q},\hat{p})\in\mathbb{R}^{d}\times\mathbb{R}^{d}\mid\hat{q}\in T_{q}M,\,(I_{d}-\varPi_{q})\hat{p}=h_{q}(\hat{q},p)\right\} 
\]

Given this, we obtain : 
\begin{prop}
\label{prop:-is-invariant}$TM$ is invariant by the Hamiltonian flow
if and only if $\forall(q,p)\in TM,$: 
\[
\nabla_{p}H(q,p)\in T_{q}M,\,(I_{d}-\varPi_{q})\nabla_{q}H(q,p)=-h_{q}(\nabla_{p}H(q,p),p)
\]
\end{prop}

If $I$ is an interval and $\gamma:I\rightarrow M$ a smooth path
on $M$, we say that $X$ is a normal vector field along $\gamma$
if $\forall t\in I$, $X(t)\in(T_{\gamma(t)}M)^{\perp}$. We recall
the following classical formula for the covariant derivative of a
normal vector field along a path. 
\begin{prop}[Gauss-Weingarten Formula]

For $q\in M$, we define 
\begin{align*}
h_{q}^{\star}:T_{q}M\times(T_{q}M)^{\bot}\longrightarrow T_{q}M\\
(a,b)\longrightarrow d\varPi_{q}(a)b
\end{align*}

Then we have

- $\forall q\in M$ $\forall a\in T_{q}M$ $h_{q}^{\star}(a,.)$ is
the adjoint operator of $h_{q}(a,.)$ i.e : 
\[
\forall b\in(T_{q}M)^{\bot},\,\forall c\in T_{q}M\,<c,h_{q}^{\star}(a,b)>=<h_{q}(a,c),b>
\]

- If $X$ is a normal vector field along $\gamma$ : 
\[
\varPi_{\gamma(t)}\left(\frac{dX(t)}{dt}\right)=-h_{\gamma(t)}^{\star}\left(\frac{d\gamma(t)}{dt},X(t)\right)
\]
\end{prop}

\subsection{\label{subsec:Projection-on-a}Projection on a submanifold $M$ and
its differential}

Since $M$ is a closed subset of $\mathbb{R}^{d}$, the distance function
$d(.,M)$ is Lipschitz continuous. Moreover, if we consider the normal
bundle $NM=\underset{q\in M}{\bigcup}\{q\}\times(T_{q}M)^{\bot}$,
as above the Tubular Neighborhood Theorem ensures there exists a continuous
positive function on $M$, $\theta:\,q\rightarrow\theta(q)$ such
that if $N_{\theta}=\{(q,n)\in NM/\left\Vert n\right\Vert <\theta(q)\}$
then : 
\begin{eqnarray*}
\psi: & N_{\theta}\longrightarrow\mathbb{R}^{d}\\
 & (q,n)\longrightarrow q+n
\end{eqnarray*}

is a smooth diffeomorphism onto its open image $U_{\theta}=\psi(N_{\theta})$.

If $\pi_{1}:\,NM\longrightarrow M$ ~is the canonical projection,
then for $q\in U_{\theta}$, $\tilde{q}=\pi_{1}\circ\psi^{-1}(q)$
is the unique point in $M$ such that $d(q,M)=\left\Vert q-\tilde{q}\right\Vert $.
This shows that $q\rightarrow d(q,M)$ is smooth on the open set $U_{\theta}\setminus M$.

On $U_{\theta}$, we can define the projection on $M$ as : 
\begin{align*}
\pi_{M}:U_{\theta} & \rightarrow M\\
q & \rightarrow\tilde{q}=\pi_{1}\circ\psi^{-1}(q)
\end{align*}

We denote as before $v(q)=d\pi_{M}(q)_{|T_{\tilde{q}}M}$ . We know
that $v(q)$ is an automorphism of $T_{\tilde{q}}M$. Moreover: 
\begin{prop}[Properties of v(q)]
\label{prop:v inv q}

$\forall q\in U_{\theta}$ , $\forall p\in T_{\tilde{q}}M$, $v^{-1}(q)(p)-p=-h_{\tilde{q}}^{\star}(p,q-\tilde{q})$
with $\tilde{q}=\pi_{M}(q)$ . Moreover, $v(q)$ is a self-adjoint
automorphism of $T_{\tilde{q}}M$ i.e $\forall(p_{1},p_{2})\in\left(T_{\tilde{q}}M\right)^{2}$:
\[
<v(q)(p_{1}),p_{2}>=<p_{1},v(q)(p_{2})>
\]
\end{prop}

\begin{proof}
Let $q\in U_{\theta}$, $\tilde{q}=\pi_{M}(q)$, $p\in T_{\tilde{q}}M$
, $\alpha>0$ and $\gamma:(-\alpha,\alpha)\rightarrow M$ a smooth
path such that $\gamma(0)=\tilde{q}$, $\dot{\gamma}(0)=p$. We define
\begin{align*}
X:(-\alpha,\alpha) & \rightarrow\mathbb{R}^{d}\\
t & \rightarrow\gamma(t)+q-\tilde{q}-\varPi_{\gamma(t)}(\gamma(t)+q-\tilde{q})
\end{align*}

So $X$ is a normal vector field along $\gamma$. Moreover for $0<\alpha^{'}<\alpha$
sufficiently small, $\forall t\in(-\alpha^{'},\alpha^{'})$, $\gamma(t)+X(t)\in U_{\theta}$
by continuity. We have then : 
\[
\forall t\in(-\alpha^{'},\alpha^{'}),\pi_{M}\left(\gamma(t)+X(t)\right)=\gamma(t)
\]

By differentiating at $t=0$ : 
\[
v(q)(p+\varPi_{\tilde{q}}\dot{X}(t))=p
\]

By the Gauss-Weingarten formula applied to $X$, we obtain : 
\[
v^{-1}(q)(p)-p=-h_{\tilde{q}}^{\star}(p,q-\tilde{q})
\]

To prove the last point, let $(a,b)\in\left(T_{\tilde{q}}M\right)^{2}$,
we have : 
\begin{align*}
<v^{-1}(q)(a),b>-<a,b> & =-<h_{\tilde{q}}^{\star}(a,q-\tilde{q}),b>\\
 & =-<q-\tilde{q},h_{\tilde{q}}(a,b)>
\end{align*}

Then by symmetry of $h_{\tilde{q}}$ : 
\begin{align*}
<v^{-1}(q)(a),b>-<a,b> & =-<q-\tilde{q},h_{\tilde{q}}(b,a)>\\
 & =-<h_{\tilde{q}}^{\star}(b,q-\tilde{q}),a>\\
 & =<v^{-1}(q)(b),a>-<b,a>
\end{align*}

So $v^{-1}(q)$ is self-adjoint and therefore $v(q)$ is too. 
\end{proof}

\subsection{Change of coordinates and the Hamilton-Jacobi equation}

We need first a simple lemma about change of variables for solutions
of the Hamilton-Jacobi equation, that allows us to characterize viscosity
solutions on a submanifold $M$. 
\begin{lem}
\label{lem:Change-of-variables}Change of variables

1.Let $\phi:U\rightarrow V$ be a smooth diffeomorphism between open
sets $U,\,V\subset\mathbb{R}^{d}$. We define $\hat{H}$ on $U\times\mathbb{R}^{d}$
by $\hat{H}(q,p)=H(\hat{q},\,^{t}d\phi^{-1}(\hat{q}),p)$ where $\hat{q}=\phi(q)$
and $^{t}d\phi^{-1}(\hat{q})$ is the transpose of the matrix $d\phi^{-1}(\hat{q})$.
And if $u:\,I\times V\rightarrow\mathbb{R}$ is a function, we define
$\hat{u}$ on $I\times U$ as its pullback by $\phi$ with respect
to the second coordinate (i.e $\hat{u}(t,q)=u(t,\phi(q))$. We have
then:

$\hat{u}$ is a classical solution $(t_{0},q_{0})\in I\times U$ (respectively
a viscosity solution) of the Hamilton-Jacobi equation with Hamiltonian
$\hat{H}$ if and only if $u$ is a classical solution at $(t_{0},\phi(q_{0}))\in I\times V$
(respectively a viscosity solution) of the Hamilton-Jacobi equation
with Hamiltonian $H$.

2. if $H$ is a Hamiltonian function and $M$ a smooth submanifold
of $\mathbb{R}^{d}$, let $q_{0}\in M$:

$\forall p\in\mathbb{R}^{d}$ $H(q_{0},p)=H(q_{0},\Pi_{q_{0}}(p))$
~if and only if there exists open sets $U,\,V\subset\mathbb{R}^{d}$,
with $q_{0}\in V$ and an adapted chart $\phi:U\rightarrow V$ with
$\phi(0)=q_{0}$, such that $\hat{H}$ defined as above verifies $\forall p\in\mathbb{R}^{d}$~$\hat{H}(0,p)=\hat{H}(0,p_{T})$
where $p_{T}$ is the orthogonal projection of $p$ on $\mathbb{R}^{m}\times\{0_{d-m}\}$. 
\end{lem}

\begin{proof}
1. In case of a classical solution, since $u$ is differentiable at
$(t_{0},\phi(q_{0}))\in I\times V$ then by the chain rule : 
\[
d_{q}u(t_{0},\hat{q}_{0})\circ d\phi(q_{0})=d_{q}\hat{u}(t_{0},q_{0})
\]

Let $j:\mathbb{R}^{d}\rightarrow(\mathbb{R}^{d})^{\star}$ be the
canonical isomorphism associated to the canonical scalar product.
Then for $p\in(\mathbb{R}^{d})^{\star}$ and $A$ a matrix of $\mathbb{R}^{d\times d}$
identified to its associated endomorphism of $\mathbb{R}^{d}$, we
have 
\[
j^{-1}(p\circ A)=^{t}A(j^{-1}(p))
\]

Applying to the above : 
\[
^{t}d\phi(q_{0})\nabla_{q}u(t_{0},\hat{q}_{0})=\nabla_{q}\hat{u}(t_{0},q_{0})
\]

Because $u$ solves the Hamilton-Jacobi equation at $(t_{0},q_{0})$
with Hamiltonian $H$, we have

\begin{eqnarray*}
\partial_{t}u(t_{0},\phi(q_{0})) & = & \partial_{t}\hat{u}(t_{0},q_{0})=-H(\phi(q_{0}),\nabla_{q}u(t_{0},\phi(q_{0}))\\
 & = & -\hat{H}(q_{0},^{t}d\phi(q_{0})\nabla_{q}u(t_{0},\hat{q}_{0})\\
 & = & -\hat{H}(q_{0},\nabla_{q}\hat{u}(t_{0},q_{0}))
\end{eqnarray*}

For viscosity solutions, the proof is exactly the same by using test
functions instead.

2. For the if part : If $U,\,V$ are open subsets of $\mathbb{R}^{d}$
and $\phi:U\rightarrow V$ is an adapted local chart for $M$ around
$q_{0}\in V\bigcap M$. Moreover for $p\in\mathbb{R}^{d}$, let $p_{T}$
be its orthogonal projection on $\mathbb{R}^{m}\times\{0_{d-m}\}$,
then for any $y\in T_{q_{0}}M$ : 
\[
<y,^{t}d\phi^{-1}(q_{0})p>=<d\phi^{-1}(q_{0})y,p>
\]

Since $d\phi^{-1}(q_{0})$ is an isomorphism between $T_{q_{0}}M$
and $\mathbb{R}^{m}\times\{0_{d-m}\}$, we obtain that $d\phi^{-1}(q_{0})y\in\mathbb{R}^{m}\times\{0_{d-m}\}$
and therefore : 
\[
<y,^{t}d\phi^{-1}(q_{0})p>=<d\phi^{-1}(q_{0})y,p_{T}>=<y,^{t}d\phi^{-1}(q_{0})p_{T}>
\]

This being true for all $y\in T_{q_{0}}M$, it shows that : 
\[
\varPi_{q_{0}}({}^{t}d\phi^{-1}(q_{0})p)=\varPi_{q_{0}}({}^{t}d\phi^{-1}(q_{0})p_{T})
\]

From this, by definition of $\hat{H}$ and by assumption on $H$ :
\[
\hat{H}(0,p)=H(q_{0},{}^{t}d\phi^{-1}(q_{0})p)=H(q_{0},{}^{t}d\phi^{-1}(q_{0})p_{T})=\hat{H}(0,p_{T})
\]

For the only if part, we can proceed exactly in the same way and swap
around in this proof the role of $H$ and $\hat{H}$. 
\end{proof}

\section{\label{sec:Restriction-of-solutions}Restriction of solutions when
$M\times\mathbb{R}^{d}$ is invariant}

\subsection{Classical solutions}

\subsubsection{Restrictions}

If $i\,:M\rightarrow\mathbb{R}^{d}$ is the canonical injection which
is a smooth embedding since $M$ is a submanifold.

We denote for $q\in M$,~the linear map $di(q)\,:T_{q}M\rightarrow T_{q}\mathbb{R}^{d}\simeq\mathbb{R}^{d}$
which is the canonical injection of $T_{q}M$ as a subvector space
of $\mathbb{R}^{d}$. Then the transpose of this linear map is $\pi_{q}:\,\left(\mathbb{R}^{d}\right)^{\star}\rightarrow T_{q}^{\star}M$
which is the map sending linear forms on $\mathbb{R}^{d}$ to their
restriction to $T_{q}M$. 
\begin{proof}
\emph{of Theorem \eqref{thm:Restrictions-and-extensions} a) in case
of classical solutions}

We consider a differentiable solution $u$ on $I\times\mathbb{R}^{d}$
of $(\mathscr{CP)}$ and $\bar{u}=u_{|I\times M}$. For $(t,q)\in I\times M$,
$\bar{u}(t,q)=u(t,i(q))$ ~where $i$ ~is the canonical embedding
$i:\,M\hookrightarrow\mathbb{R}^{d}$. We have by the chain rule and
by definition of $\pi_{q}$:

\begin{equation}
\forall\,(t,q)\in I\times M\,d_{q}\bar{u}(t,q)=d_{q}u(t,q)\circ di(q)=\pi_{q}(d_{q}u(t,q))\label{eq:diff u restr}
\end{equation}

This shows that $\forall\,(t,q)\in I\times M,\,\nabla_{q}\bar{u}(t,q)=\Pi_{q}(\nabla_{q}u(t,q))$.
Moreover, since $u$ is a solution of $(\mathscr{CP)}$, 
\[
\forall\,(t,q)\in I\times M,\,\partial_{t}\bar{u}(t,q)=\partial_{t}u(t,i(q))=-H(q,\nabla_{q}u(t,q))
\]

By the invariance of $M\times\mathbb{R}^{d}$, $H(q,\nabla_{q}u(t,q))=H(q,\Pi_{q}(\nabla_{q}u(t,q)))$
, which by definition of $\bar{H}$ : 
\[
\forall\,(t,q)\in I\times M,\,\partial_{t}\bar{u}(t,q)=-\bar{H}(q,\nabla_{q}\bar{u}(t,q))
\]
\end{proof}

\subsubsection{Extensions}
\begin{proof}
\emph{of Theorem \eqref{thm:Restrictions-and-extensions} b) in case
of classical solutions}

We have by the chain rule that 
\[
\partial_{q}u(t,q)=\partial_{q}\bar{u}(t,\pi_{M}(q)\circ d(\pi_{M}(q))+2a<q-\tilde{q},.>
\]
So 
\[
\partial_{q}u(t,q)|_{T_{\tilde{q}}M}=\partial_{q}\bar{u}(t,\pi_{M}(q))\circ v(q)
\]

Therefore : 
\begin{align}
\varPi_{\tilde{q}}\nabla u(t,q) & =v^{\ast}(q)\nabla_{q}\bar{u}(t,\pi_{M}(q))\nonumber \\
 & =v(q)\nabla_{q}\bar{u}(t,\pi_{M}(q))\,\tag{\ensuremath{v(q)}is self-adjoint}\label{eq:grad u on TqM}
\end{align}

We compute then: 
\begin{eqnarray*}
\partial_{t}u(t,q) & = & \partial_{t}\bar{u}(t,\pi_{M}(q))\\
 & = & -\bar{H}(\tilde{q},\nabla_{q}\bar{u}(t,\tilde{q}))\\
 & = & -\bar{H}(\tilde{q},v^{-1}(q)\varPi_{\tilde{q}}\nabla u(t,q))\\
 & = & -H(q,\nabla u(t,q))
\end{eqnarray*}
\end{proof}

\subsection{Viscosity solutions}

\subsubsection{Restrictions}
\begin{proof}
\emph{of Theorem (\ref{thm:Restrictions-and-extensions}) a)}

By choosing a local chart on $M$ and by using the lemma (\ref{lem:Change-of-variables}),
we know that the assumption of independence on normal moments is invariant
by change of variables. It is sufficient to consider the case $M=\mathbb{R}^{m}\times\{0_{d-m}\}$
with a Hamiltonian $H$ verifying for some $q_{0}\in M$ : 
\begin{equation}
\forall p\in\mathbb{R}^{d}\,H(q_{0},p)=H(q_{0},\pi_{M}(p))\label{eq:invariance normal moments at qo}
\end{equation}

We are going to show that if $u$ is a viscosity solution for the
Hamilton-Jacobi equation with Hamiltonian $H$ in a neighborhood of
$(t_{0},q_{0})\in\mathbb{R}\times M$, then $\bar{u}=u|_{\mathbb{R}\times M}$
is s a viscosity solution for the Hamilton-Jacobi equation with Hamiltonian
$\bar{H}$ at $(t_{0},q_{0})$.

Let us consider $\phi_{1}:\mathbb{R}\times M\longrightarrow\mathbb{R}$
of class $C^{1}$ such that $\phi_{1}-u$ has a strict local minimum
at $(t_{0},q_{0})$ and ~$\phi_{1}(t_{0},q_{0})=u(q_{0})$ . Let
$r>0$ such that if $B_{m+1}((t_{0},q_{0}),r)$ is the Euclidean ball
in $\mathbb{\mathbb{R}\times R}^{m}\times\{0_{d-m}\}$ centered at
$(t_{0},q_{0})$ of radius $r>0$ then $\phi_{1}>\bar{u}$, on $B_{m+1}((t_{0},q_{0}),r)$$\setminus\{(t_{0},q_{0})\}$.

If $q\in\mathbb{R}^{d}$, we denote $q_{M}=\pi_{M}(q)$ and $q_{d}=q-\pi_{M}(q)$,
where $\pi_{M}$ is the orthogonal projector on $M$.

$u$ is Lipschitz on $B_{d+1}((t_{0},q_{0}),1)$ and let $L$ be its
Lipschitz constant.

We consider $\epsilon>0$ with $\epsilon<\min(1,r)$, and we define
then the $C^{1}$ function $\phi_{\epsilon}:U=B_{m+1}((t_{0},q_{0}),r)\times\mathbb{R}^{d-m}\longrightarrow\mathbb{R}$
as : 
\[
\phi_{\epsilon}(t,q)=\phi_{1}(t,q_{M})+\frac{2L+\alpha(\epsilon)}{\epsilon}\left\Vert q_{d}\right\Vert ^{2}
\]

Where $\alpha(\epsilon)=\frac{\epsilon L^{2}}{4\mu(\epsilon)}$ and
$\mu(\epsilon)=\inf\left\{ \phi_{1}(t,q)-u(t,q)\,/q\in M,\left\Vert (t,q)\right\Vert \leq\epsilon,\left\Vert (t,q)\right\Vert \geq\frac{\sqrt{3}}{2}\epsilon\right\} $

Let $m_{\epsilon}=\inf\{\phi_{\epsilon}(t,q)-u(t,q)/(t,q)\in\bar{B}_{d}((t_{0},q_{0}),\epsilon)\}$
. By compacity of $\bar{B}_{d}((t_{0},q_{0}),\epsilon)$ , we have
$(t_{\epsilon},q_{\epsilon})\in\bar{B}_{d}((t_{0},q_{0}),\epsilon)$
that verifies $\phi_{\epsilon}(t_{\epsilon},q_{\epsilon})-u(t_{\epsilon},q_{\epsilon})=m_{\epsilon}$.
Notice that $m_{\epsilon}\leq0$, since $\phi_{\epsilon}(t_{0},q_{0})=u(t_{0},q_{0})$

Then either $q_{\epsilon}\in M$ and in this case $q_{\epsilon}=q_{0}$,
$t_{\epsilon}=t_{0}$, because $\phi_{1}-\bar{u}$ has a attains its
minimum on a unique point in $\bar{B}_{m}((t_{0},q_{0}),\epsilon)$.
So in this case $\phi_{\epsilon}-u$ has a local minimum at $(t_{0},q_{0})$
so $\phi_{\epsilon}$ is a sub-solution at $(t_{0},q_{0})$ of the
Hamilton-Jacobi with Hamiltonian $H$. Since $\pi_{M}(\nabla_{q}\phi_{\epsilon}(t_{0},q_{0}))=\nabla_{q}\phi_{1}(t_{0},q_{0})$
, we obtain that $\phi_{1.}$is a sub-solution at $(t_{0}$ of the
stationary Hamilton-Jacobi with Hamiltonian $\bar{H}$.

If $q_{\epsilon}\notin M$ , we show first that $\left\Vert (t_{\epsilon},q_{\epsilon})-(t_{0},q_{0})\right\Vert <\epsilon$.
We denote $q_{\epsilon,M}=\pi_{M}(q_{\epsilon})$ and $q_{\epsilon,N}=q_{\epsilon}-\pi_{M}(q_{\epsilon})$.
First of all we have : 
\begin{eqnarray*}
m_{\epsilon} & = & \phi_{1}(t_{\epsilon},q_{\epsilon,M})-u(t_{\epsilon},q_{\epsilon,M})+u(t_{\epsilon},q_{\epsilon,M})-u(t_{\epsilon},q_{\epsilon})+\frac{2L+\alpha(\epsilon)}{\epsilon}\left\Vert q_{\epsilon,N}\right\Vert ^{2}\\
0 & > & -L\left\Vert q_{\epsilon,N}\right\Vert +\frac{2L+\alpha(\epsilon)}{\epsilon}\left\Vert q_{\epsilon,N}\right\Vert ^{2}
\end{eqnarray*}

So $\left\Vert q_{\epsilon,N}\right\Vert <\frac{\epsilon L}{2L+\alpha}<\frac{\epsilon}{2}$.
Also we must have $\left\Vert (t_{\epsilon},q_{\epsilon,M})\right\Vert <\frac{\sqrt{3}}{2}\epsilon$
, if not we would have : 
\begin{eqnarray*}
0>m_{\epsilon} & \geq & \mu(\epsilon)-L\left\Vert q_{\epsilon,N}\right\Vert +\frac{2L+\alpha}{\epsilon}\left\Vert q_{\epsilon,N}\right\Vert ^{2}\\
 & \geq & \mu(\epsilon)-L\frac{\epsilon L}{2(2L+\alpha)}+\frac{2L+\alpha}{4\epsilon}\left(\frac{\epsilon L}{2L+\alpha}\right)^{2}\\
 & \geq\frac{\epsilon L^{2}}{4\alpha} & -\frac{\epsilon L^{2}}{4(2L+\alpha)}>0
\end{eqnarray*}

which is impossible. Therefore $\left\Vert q_{\epsilon}-q_{0}\right\Vert <\epsilon$
and $\phi_{\epsilon}-u$ has a local minimum at an interior point
$(t_{\epsilon},q_{\epsilon})\in B_{d+1}((t_{0},q_{0}),\epsilon)$.

Since $u$ is a viscosity solution, $\phi_{\epsilon}$ is a sub-solution
at $(t_{\epsilon},q_{\epsilon})$ of the Hamilton-Jacobi equation,
i.e $\partial_{t}\phi_{1}(t_{\epsilon},q_{\epsilon})+H(q_{\epsilon},\nabla_{q}\phi_{\epsilon}(t_{\epsilon},q_{\epsilon}))\leq0$
. We also have 
\[
\nabla_{q}\phi_{\epsilon}(t_{\epsilon},q_{\epsilon})=\nabla_{q}\phi_{1}(t_{\epsilon},q_{\epsilon,M})+\frac{4L+2\alpha(\epsilon)}{\epsilon}q_{\epsilon,N}
\]

Therefore we can construct a sequence $(t_{p},q_{p})\in B_{d+1}((t_{0},q_{0}),\frac{1}{p})$
such that $\partial_{t}\phi_{1}(t_{p},q_{p})+H(q_{p},\nabla_{q}\phi_{\frac{1}{p}}(t_{p},q_{p}))\leq0$
and 
\[
\nabla_{q}\phi_{\frac{1}{p}}(t_{p},q_{p})=\nabla_{q}\phi_{1}(t_{p},q_{p,M})+2(2L+\alpha_{p})pq_{p,N}
\]

But as seen above $(2L+\alpha_{p})p\left\Vert q_{p}-\pi_{M}(q_{p})\right\Vert \leq L$
, so it has convergent subsequence to some $n\in N$ and at the limit
\[
\partial_{t}\phi_{1}(t_{0},q_{0})+H(q_{0},\nabla_{q}\phi_{1}(t_{0},q_{0})+n)\leq0
\]

Since $H(q_{0},\nabla_{q}\phi_{1}(t_{0},q_{0})+n)=H(q_{0},\nabla_{q}\phi_{1}(t_{0},q_{0}))$
by assumption on $H$ , we obtain that $\phi_{1.}$is a sub-solution
at $(t_{0},q_{0})$ of Hamilton-Jacobi with Hamiltonian $\bar{H}$.

This completes the proof. The case of supersolutions can be treated
in the same way by considering with the same notations as above $\phi_{\epsilon}(t,q)=\phi_{1}(t,q_{M})-\frac{2L+\alpha(\epsilon)}{\epsilon}\left\Vert q_{d}\right\Vert ^{2}$
and where $\alpha(\epsilon)=\frac{\epsilon L^{2}}{4\mu(\epsilon)}$
and $\mu(\epsilon)=\inf\left\{ u(t,q)-\phi_{1}(t,q)/q\in M,\left\Vert (t,q)\right\Vert \leq\epsilon,\,\left\Vert (t,q)\right\Vert \geq\frac{\sqrt{3}}{2}\epsilon\right\} $ 
\end{proof}

\subsubsection{Extensions}
\begin{proof}
\emph{of Theorem \eqref{thm:Restrictions-and-extensions} b)}

Let $\phi:\mathbb{R}\times\mathbb{R}^{d}\rightarrow\mathbb{R}$ be
a $C^{1}$ test function such that $\phi\geq u$ in a neighborhood
of $(t_{0},q_{0})$, that is equal to $0$ at $(t_{0},q_{0})$. When
$q_{0}\in M$, we notice that $\phi_{|\mathbb{R}\times M}-\bar{u}$
has then a local minimum at $(t_{0},q_{0})$and therefore : 
\[
\partial_{t}\phi(t_{0},q_{0})+\bar{H}(q_{0},\varPi_{q_{0}}\nabla_{q}\phi(t_{0},q_{0}))\leq0
\]

By definition of $H$ , $\bar{H}(q_{0},\varPi_{q_{0}}\nabla\phi(t_{0},q_{0}))=H(q_{0},\nabla\phi(t_{0},q_{0}))$,
so:

\[
\partial_{t}\phi(t_{0},q_{0})+H(q_{0},\varPi_{q_{0}}\nabla_{q}\phi(t_{0},q_{0}))\leq0
\]

Now if $q_{0}\in U_{\theta}\setminus M$ then $d(q_{0},M)\neq0$ and
$q\rightarrow d(q,M)$ is smooth in a neighborhood of $q_{0}$. We
consider~$r>0$ such that the open ball $B(q_{0},r)\subset U_{\theta}\bigcap M^{c}$
and $\phi\geq u$ in $[t_{0}-r,t_{0}+r]\times B(q_{0},r)$.

Let $Q_{0}=\pi_{M}(q_{0})$, we consider $\varphi:U\supset\{Q_{0}\}\longrightarrow V$
a local chart i.e a smooth diffeomorphim between the open sets $U,\,V\subset\mathbb{R}^{d}$
such that $\psi(U\bigcap M)=V\bigcap(\mathbb{R}^{m}\times\{0_{d-m}\}$
and we denote $\varphi=(\varphi_{1},..,\varphi_{d})$.

We consider then the neighborhood $U_{1}=\pi_{M}(B(q_{0},r))\bigcap U\bigcap M$
of $Q_{0}$ in $M$.

We write : $q_{0}-Q_{0}=\stackrel[i=0]{d-m-1}{\sum}\alpha_{i}\nabla\varphi_{d-i}(Q_{0})$.
We consider the map on $U_{1}$ defined by : 
\[
g:Q\rightarrow Q+\stackrel[i=0]{d-m-1}{\sum}\alpha_{i}\nabla\varphi_{d-i}(Q)
\]
And we set $U_{2}=g^{-1}(B(q_{0},r))$ which is a non-empty (because
it contains $Q_{0}$), open neighborhood of $Q_{0}$ in $M$. Also
it is easy to verify that $g$ is injective.

We define then $h:U_{2}\rightarrow\mathbb{R}^{d}$ by 
\[
h(Q)=Q+(g(Q)-Q)\frac{d(Q_{0},M)}{d(Q,M)}
\]

We have that $h(Q_{0})=q_{0}$ , $h$ is smooth, injective and for
$Q\in U_{2}$, $h(Q)-Q\in T_{Q}M$ and $d(h(Q),M)=\left\Vert h(Q)-Q\right\Vert =\left\Vert Q_{0}-q_{0}\right\Vert $.
Also we recall that since $(t_{0},q_{0})$ is a critical point of
$u-\phi$, and \uline{\mbox{$q\rightarrow u(t_{0},q)$}} is differentiable
along vectors belonging to $(T_{Q_{0}}M)^{\bot}$: 
\begin{equation}
\varPi_{Q_{0}}^{\bot}\nabla_{q}u(t_{0},q_{0})=\varPi_{Q_{0}}^{\bot}\nabla_{q}\phi(t_{0},q_{0})=2a(q_{0}-Q_{0})\label{eq:normal gradient of test function}
\end{equation}

We define now the following $C^{1}$ function in $[t_{0}-r,t_{0}+r]\times U_{2}$:

\begin{equation}
\bar{\phi}(t,Q)=\phi(t,h(Q))-a\left\Vert h(Q)-Q\right\Vert ^{2}\label{eq:definition phibar}
\end{equation}

First we notice that if for $q\in B(q_{0},r)$, $\phi_{1}(t,q)=\phi(t,q)-a\left\Vert q-\pi_{M}(q)\right\Vert ,$then
by (\ref{eq:normal gradient of test function}) : 
\begin{align}
\varPi_{Q_{0}}^{\bot}\nabla_{q}\phi_{1}(t_{0},q_{0}) & =0\label{eq:zero normal grad phi1}\\
\varPi_{Q_{0}}\nabla_{q}\phi_{1}(t_{0},q_{0}) & =\varPi_{Q_{0}}\nabla_{q}\phi(t_{0},q_{0})
\end{align}

And that : 
\[
\bar{\phi}(t,Q)=\phi_{1}(t,h(Q))=\phi(t,h(Q))-a\left\Vert Q_{0}-q_{0}\right\Vert ^{2}
\]

Also :$\bar{\phi}(t_{0},Q_{0})=\bar{u}(t_{0},Q_{0})$ and for $(t,Q)\in[t_{0}-\beta,t_{0}+\beta]\times U_{2}$:
\[
\bar{\phi}(t,Q)\geq u(t,h(Q))-a\left\Vert Q_{0}-q_{0}\right\Vert ^{2}=\bar{u}(t,Q)
\]

And so $\bar{\phi}-\bar{u}$ has a minimum in $[t_{0}-\beta,t_{0}+\beta]\times U_{2}$
at $(t_{0},Q_{0})$.

Since $\bar{u}$ is a viscosity solution of the Hamilton-Jacobi equation
with Hamiltonian $\bar{H}$, $\bar{\phi}$ is a sub-solution at $(t_{0},Q_{0})$
: 
\begin{equation}
\partial_{t}\bar{\phi}(t_{0},Q_{0})+\bar{H}(Q_{0},\nabla_{q}\bar{\phi}(t_{0},Q_{0}))\leq0\label{eq: phi bar sub sol}
\end{equation}

Now we observe that since $\pi_{M}(h(Q))=Q$ , we have 
\[
d\pi_{M}(h(Q))\circ dh(Q)=Id_{T_{Q}M}
\]

So : 
\[
d\pi_{M}(h(Q))\circ\varPi_{Q}\circ dh(Q)=Id_{T_{Q}M}
\]

Therefore : 
\[
v^{-1}(Q)=\varPi_{Q}\circ dh(Q)
\]

Now by differentiating (\ref{eq:definition phibar}) with respect
to $Q$ : 
\[
\partial_{q}\bar{\phi}(t,Q)=\partial_{q}\phi_{1}(t,h(Q))\circ dh(Q)
\]

So because (\ref{eq:zero normal grad phi1}), for $y\in T_{Q_{0}}M$
\begin{align*}
<\nabla_{q}\bar{\phi}(t,Q_{0}),y> & =<\varPi_{Q_{0}}\nabla_{q}\phi(t,q_{0}),dh(Q_{0})(y)>\\
 & =<\varPi_{Q_{0}}\nabla_{q}\phi(t,q_{0}),\varPi_{Q_{0}}dh(Q_{0})(y)>\\
 & =<\varPi_{Q_{0}}\nabla_{q}\phi(t,q_{0}),v^{-1}(Q_{0})(y)>\\
 & =<v^{-1}(Q_{0})(\varPi_{Q_{0}}\nabla_{q}\phi(t,q_{0})),y>
\end{align*}

Where the last equation follows from the self-adjointness of $\varPi_{Q_{0}}$
and $v^{-1}(Q_{0})$.

Hence : 
\begin{align*}
\nabla_{q}\bar{\phi}(t,Q_{0}) & =v^{-1}(Q_{0})\left(\varPi_{Q_{0}}\nabla_{q}\phi(t,q_{0})\right)
\end{align*}

\begin{equation}
\partial_{q}\phi(t,q_{0})\circ v^{-1}(q_{0})=\eta(q_{0})\partial_{q}\bar{\phi}(t,Q_{0})\label{eq:partial der phi}
\end{equation}

Substituting in (\ref{eq: phi bar sub sol}) and using the definition
of $H$ , we deduce that$\phi$ is a sub-solution at $(t_{0},q_{0})$
for the Hamilton-Jacobi equation with Hamiltonian $H$.

The case of super-solutions is treated symmetrically. 
\end{proof}

\bibliographystyle{amsalpha}
\bibliography{biblio.bib}

\end{document}